\documentclass{amsart}

\usepackage{amsmath, amssymb,epic,graphicx,mathrsfs,enumerate}

\usepackage{amsfonts}
\usepackage{amsthm}
\usepackage{amssymb}
\usepackage{latexsym}
\usepackage{longtable}
\usepackage{epsfig}
\usepackage{amsmath}
\usepackage{hhline}

\input xy
\xyoption{all}

\newtheorem{defi}{Definition}
\newtheorem{teor}{Theorem}
\newtheorem{cor}{Corollary}
\newtheorem{prop}{Proposition}
\newtheorem{obs}{Observation}
\newtheorem{lemma}{Lemma}
\newtheorem{nota}{Notations}
\DeclareMathOperator{\Sym}{Sym}
\DeclareMathOperator{\Aut}{Aut}
\DeclareMathOperator{\soc}{soc}
\DeclareMathOperator{\Stab}{Stab}

\begin{document}

\title{Covering certain monolithic groups \\ with proper subgroups}
\author{Martino Garonzi}
\date{27th of August, 2011}

\begin{abstract}
Given a finite non-cyclic group $G$, call $\sigma(G)$ the least number of proper subgroups of $G$ needed to cover $G$. In this paper we give lower and upper bounds for $\sigma(G)$ for $G$ a group with a unique minimal normal subgroup $N$ isomorphic to $A_n^m$ where $n \geq 5$ and $G/N$ is cyclic. We also show that $\sigma(A_5 \wr C_2)=57$.
\end{abstract}

\maketitle

\section{Introduction} \label{intro}

Given a finite non-cyclic group $G$, call $\sigma(G)$ the least number of proper subgroups of $G$ needed to cover $G$ set-theoretically. This notion has been introduced the first time by Cohn in 1994 in \cite{cohn}. We usually call ``cover'' of $G$ a family of proper subgroups of $G$ which covers $G$, and ``minimal cover'' of $G$ a cover of $G$ consisting of exactly $\sigma(G)$ elements. If $G$ is cyclic then $\sigma(G)$ is not well defined because no proper subgroup contains any generator of $G$; in this case we define $\sigma(G)=\infty$, with the convention that $n < \infty$ for every integer $n$. In \cite{tomkinson} Tomkinson showed that if $G$ is a finite solvable group then $\sigma(G)=q+1$, where $q$ is the least order of a chief factor of $G$ with more than one complement. The behavior of the function $\sigma$ has been intensively studied for the almost simple groups. The alternating and symmetric groups have been considered by Mar\'oti in \cite{maroti2}. In \cite{psl} Britnell, Evseev, Guralnick, Holmes and Mar\'oti studied the linear groups $GL(n,q)$, $PGL(n,q)$, $SL(n,q)$, $PSL(n,q)$. In \cite{lucido} Lucido studied the Suzuki groups. In \cite{kneser} Lucchini and Mar\'oti found an asymptotic formula for the function which assigns to the positive integer $x$ the number of positive integers $n$ at most $x$ with the property that $\sigma(S)=n$ for some non-abelian simple group $S$.

If $N$ is a normal subgroup of a finite group $G$ then $\sigma(G) \leq \sigma(G/N)$, since every cover of $G/N$ can be lifted to a cover of $G$. We say that $G$ is ``$\sigma$-primitive'' if $\sigma(G)<\sigma(G/N)$ for every non-trivial normal subgroup $N$ of $G$. Since every finite group has a $\sigma$-primitive epimorphic image with the same $\sigma$, the structure of the $\sigma$-primitive groups is of big interest. It was studied by Lucchini and Detomi in \cite{lucchini}. They proved for instance that every $\sigma$-primitive group is a subdirect product of monolithic groups (i.e. groups with only one minimal normal subgroup). This and other partial results lead us to believe that the monolithic groups have a crucial role in this story. In the same paper Lucchini and Detomi conjectured that every non-abelian $\sigma$-primitive group is monolithic. This motivates us in the study of the function $\sigma$ for the monolithic $\sigma$-primitive groups.

Let us consider a monolithic $\sigma$-primitive group $G$. If $\soc(G)$ is abelian then it is easy to prove that $\soc(G)$ is complemented in $G$ and $\sigma(G)=c+1$, where $c$ is the number of complements of $\soc(G)$ in $G$. Let now $n,m$ be positive integers with $n \geq 5$. Suppose that $\soc(G)=A_n^m$ and that $G/\soc(G)$ is cyclic. Write $\soc(G)=T_1 \times \cdots \times T_m = T^m$, with $T=A_n$, and define $X:=N_G(T_1)/C_G(T_1)$. Then either $X \cong A_n$ (``even case'') or $X \cong S_n$ (``odd case''). In the even case $G \cong A_n \wr C_m$ (cfr. \cite{spagn}, Definition 1.1.8 and Remark 1.1.40.13). These groups have been studied in \cite{margar} obtaining lower and upper bounds for $\sigma(G)$ and its exact value in the case $n \equiv 2 \mod(4)$.

Consider now the odd case. Let $\gamma \in G$ be such that $\gamma \soc(G)$ generates $G/\soc(G)$, so that $G = \langle T^m,\gamma \rangle$. Since $$T^m < G \leq \Aut(T^m) \cong T \wr \Sym(m),$$every element of $G$ has the form $(x_1,\ldots,x_m) \gamma^k$ with $x_1,\ldots,x_m \in T$ and $k$ an integer. Moreover $\gamma$ itself is of the form $(y_1,\ldots,y_m) \delta$ with $y_1,\ldots,y_m \in \Aut(A_n)$, and $\delta \in \Sym(m)$ is an $m$-cycle since $G$ acts transitively on the $m$ factors of $T^m$. $\gamma$ can be chosen in such a way that each $y_i$ is either $1$ or equal to $\tau:=(12) \in S_n-A_n$. Since we are in the odd case the number of indices $i \in \{1,\ldots,m\}$ such that $y_i=\tau$ is odd. It is easy to show that $\gamma$ is conjugate to $(1,\ldots,1,\tau)\delta$ in $G$. Therefore we may choose $\gamma$ to be $(1,\ldots,1,\tau) \delta$ and clearly it is not restrictive to choose $\delta:=(1 \cdots m)$. It turns out that $G$ is the semidirect product $$A_n^m \rtimes \langle \gamma \rangle.$$

Let us fix some notation. Let $C:=C_G(T_1)$. Let $U$ be a maximal subgroup of $G$ supplementing the socle $N$ of $G$. $U$ is called ``of product type'' if $U=N_G(M \times M^{a_2} \times \cdots \times M^{a_m})$ with $M$ a maximal $N_U(T_1)$-invariant subgroup of $T_1$ (cfr. \cite{spagn}, Remark 1.1.40.20) and $a_2,\ldots,a_m \in \Aut(A_n)$. In this case $M=N_U(T_1) \cap T_1$ and $N_U(T_1) C/C$ is a maximal subgroup of $N_G(T_1)/C \cong S_n$ (cfr. \cite{spagn}, Remark 1.1.40.21) whose intersection with $T_1C/C$ is $MC/C \cong M$, so that $M$ is of the form $K \cap A_n$ with $K$ maximal in $S_n$. $U$ is said to be of ``diagonal type'' if $U=N_G(\Delta)$ where $\Delta=\Delta_1 \times \cdots \times \Delta_{m/q}$, where $q$ is a prime divisor of $m$ and $\Delta_i=\{(x,x^{\alpha_{i_1}},\ldots,x^{\alpha_{i_q}})\ |\ x \in A_n\}$, where $\alpha_{i_k} \in \text{Aut}(A_n)$ for $k=1,\ldots,q$. In this case we also say that $U$ is of ``diagonal type $q$''. It turns out that every maximal subgroup of $G$ supplementing the socle is either of product type or of diagonal type.

In this paper we establish the following result, generalizing the results in \cite{maroti2} about $\sigma(S_n)$ (which corresponds to the case $m=1$). The arguments we use involve the same covers of $S_n$ considered in \cite{maroti2}, and this is why the results have similar flavour: in particular, we obtain an exact formula for $\sigma(G)$ when $n$ is odd with some exceptions, and an asymptotic formula when $n$ is even.

\begin{teor} \label{fin}
Let $m,n$ be positive integers, and let $G:=A_n \rtimes C_{2m}$ as above. Let $\omega(x)$ denote the number of prime factors of the positive integer $x$. The following holds.

\begin{enumerate}
\item Suppose that $n \geq 7$ is odd and $m \neq 1$ if $n=9$. Then $$\sigma(G) = \omega(2m) + \sum_{i=1}^{(n-1)/2} \binom{n}{i}^m.$$
\item If $n=5$ then $$10^m \leq \sigma(G) \leq \omega(2m) + 5^m + 10^m.$$If $n=5$ and every prime divisor of $m$ is either $2$ or $3$ then $$\sigma(G) = \omega(2m) + 5^m + 10^m.$$
\item Suppose that $n \geq 8$ is even. Then $$\left( \frac{1}{2} \binom{n}{n/2} \right)^m \leq \sigma(G) \leq \omega(2m) + \left( \frac{1}{2} \binom{n}{n/2} \right)^m + \sum_{i=1}^{[n/3]} \binom{n}{i}^m.$$In particular $\sigma(G) \sim \left( \frac{1}{2} \binom{n}{n/2} \right)^m$ as $n \to \infty$.
\item If $n=6$ then $$\sigma(G) = \omega(2m) + 2 \cdot 6^m.$$
\end{enumerate}

\end{teor}

Here the upper bound for $\sigma(G)$ is always given by the cardinality of a cover consisting of the $\omega(2m)$ maximal subgroups of $G$ containing its socle and suitable maximal subgroups of product type, $N_G(M \times M^{a_2} \times \cdots \times M^{a_m})$, where the $N_{S_n}(M)$'s cover $S_n-A_n$.

We also compute $\sigma(A_5 \wr C_2)$ (corresponding to the even case when $(n,m)=(5,2)$), which is not computed in \cite{margar}. Similarly as above and as in the results in \cite{margar}, a minimal cover of $A_5 \wr C_2$ consists of the maximal subgroups containing the socle and a family of subgroups of product type corresponding to a cover of $A_5$ (consisting of the normalizers of the Sylow $5$-subgroups and four point stabilizers).

\begin{teor} \label{a5wrc2}
$\sigma(A_5 \wr C_2) = 1 + 4 \cdot 5 + 6 \cdot 6 = 57$.
\end{teor}

Compare this result with the corresponding odd case: $\sigma(A_5^2 \rtimes C_{4}) = 1 + 5 \cdot 5 + 10 \cdot 10 = 126$. Note that $A_5 \wr C_2$ is the easiest example of a non-almost-simple monolithic group with non-abelian socle.

\section{Preliminary lemmas}

In the present section we collect some technical lemmas which will be useful in the next section.

Let $n$ be a positive integer and let $c_1,\ldots,c_k \in \{1,\ldots,n\}$ be such that $c_1+...+c_k=n$. A ``$(c_1,\ldots,c_k)$-cycle'' will be an element of $S_n$ which can be written as the product of $k$ pairwise disjoint cycles of length $c_1,\ldots,c_k$. An ``intransitive subgroup of $S_n$ (resp. $A_n$) of type $(c_1,\ldots,c_k)$'' will be the biggest subgroup of $S_n$ (resp. $A_n$) acting on $\{1,\ldots,n\}$ with $k$ given orbits of size $c_1,\ldots,c_k$. It is clearly isomorphic to $S_{c_1} \times \cdots \times S_{c_k}$ (resp. $(S_{c_1} \times \cdots \times S_{c_k}) \cap A_n$).

\begin{prop}[Stirling's formula]
\label{Stirling} For all positive integers $n$ we have
$$\sqrt{2\pi n} {(n/e)}^{n} e^{1/(12n+1)} < n! < \sqrt{2\pi n} {(n/e)}^{n} e^{1/(12n)}.$$
\end{prop}

The following lemma is shown in the proof of lemma 2.1 in \cite{maroti}.

\begin{lemma} \label{ab}
For a positive integer $n$ at least $8$ we have
$${((n/a)!)}^{a}a! \geq {((n/b)!)}^{b}b!$$ whenever $a$ and $b$
are divisors of $n$ with $a \leq b$.
\end{lemma}

\begin{lemma} \label{estimprim}
Let $n \neq 9,15$ be an odd positive integer, and let $a \geq 3$ be a proper divisor of $n$. Then $$\left( \frac{n-1}{2} \right)! \left( \frac{n-3}{2} \right)! \geq (n/a)!^a \cdot a!.$$
\end{lemma}

\begin{proof}
Proceed by inspection for $21 \leq n \leq 299$, using lemma \ref{ab}. Assume $n \geq 300$. Let us use Stirling's formula. We are reduced to prove that $$\sqrt{\pi (n-1)} ((n-1)/2e)^{(n-1)/2} \sqrt{\pi(n-3)} ((n-3)/2e)^{(n-3)/2} \geq$$ $$\geq 2 \sqrt{2 \pi n/a}^a (n/ae)^n \sqrt{2 \pi a} (a/e)^a.$$Using the inequalities $\pi \geq \sqrt{2 \pi}$ and $n-3 \geq a$ we are reduced to prove that $$(n-1)^{1/2} (n-1)^{(n-1)/2} (n-3)^{(n-3)/2} \geq 2/(2e)^2 \sqrt{2 \pi n/a}^a (2n/a)^n (a/e)^a,$$and using $n-1 \geq n-3$ we obtain: $$(n-1)^{1/2} (n-3)^{n-2} \geq (2/(4e^2)) (2 \pi n/a)^{a/2} (2n/a)^n (a/e)^a.$$Using the inequality $3 \leq a \leq \sqrt{n}$ we obtain: $$(n-1)^{1/2} (n-3)^{n-2} \geq (2/4e^2) (2 \pi n/3)^{\sqrt{n}/2} (2n/3)^n (\sqrt{n}/e)^{\sqrt{n}}.$$Take logarithms and divide by $n$, obtaining $$(1/2n) \log(n-1) + ((n-2)/n) \log(n-3) \geq (1/n) \log(2/4e^2)+(1/2\sqrt{n}) \log(2 \pi/3)+$$ $$+(1/2\sqrt{n}) \log(n)+\log(2n/3)+(1/\sqrt{n})\log(\sqrt{n}/e).$$Since $\sqrt{n-1} \geq 2/4e^2$ and $(1/2\sqrt{n})\log(2 \pi/3) \leq 1/\sqrt{n}$ we are reduced to show that $$\log(n-3) \geq (2/n)\log(n-3) + (1/\sqrt{n}) \log(n)+\log(2n/3).$$Since $n \geq 300$ we have that $(2/n)\log(n-3) + (1/\sqrt{n})\log(n) < 0.37$, hence it suffices to show that $\log(n-3) \geq 0.37+\log(2n/3)$, i.e. $n-3 \geq (2/3)e^{0.37} \cdot n$. This is true since $(2/3)e^{0.37} < 0.97$.
\end{proof}

\begin{cor} \label{corsizes}
Let $n \geq 11$ be an odd integer. Then the order of an intransitive maximal subgroup of $S_n$ (resp. $A_n$) is bigger than the order of any transitive maximal subgroup of $S_n$ (resp. $A_n$) different from $A_n$.
\end{cor}

\begin{proof}
The imprimitive case follows from the lemma noticing that $((n+1)/2) ! ((n-1)/2) ! \geq ((n-1)/2)! ((n-3)/2)!$, and if $n=15$ then $((n+1)/2) ! ((n-1)/2) ! \geq (n/a)!^a a!$ for $a \in \{3,5\}$. By \cite{maroti} the order of a primitive maximal subgroup of $A_n$ or $S_n$ is at most $2.6^n$ and $((n+1)/2) ! ((n-1)/2) ! \geq 2.6^n$.
\end{proof}

\begin{lemma} \label{tremezz}
Let $n,a,b$ be positive integers, with $a>b$.
\begin{enumerate}
\item Suppose $n$ is odd. Let $K$ be an intransitive maximal subgroup of $A_n$. If $(n^2-1)^a \geq 4^a e^{2(a-b)} n^{2b}$, then $|K|^{a/b} \geq |A_n|$.
\item Suppose $n$ is even. Let $K$ be a maximal imprimitive subgroup of $A_n$ of the form $(S_{n/2} \wr S_2) \cap A_n$. If $n^a \geq 2^a e^{a-b} n^b$, then $|K|^{a/b} \geq |A_n|$.
\end{enumerate}
\end{lemma}

\begin{proof}
We prove only (1), since the proof of (2) is similar. Suppose $n$ is odd. Since the smallest intransitive maximal subgroups of $A_n$ are the ones of type $((n-1)/2,(n+1)/2)$, what we have to prove is the following inequality: $$(1/2)^{a/b} ((n-1)/2)!^{a/b} ((n+1)/2)!^{a/b} \geq n!/2.$$Since $e^{\frac{a/b}{6(n-1)+1}+\frac{a/b}{6(n+1)+1}} \geq e^{1/12n}$ for every positive integer $n$, using Stirling's formula we see that it is sufficient to show that $$(1/2)^{a/b} ((n-1)/2e)^{a(n-1)/2b} \sqrt{(\pi(n-1))^{a/b}} ((n+1)/2e)^{a(n+1)/2b} \sqrt{(\pi(n+1))^{a/b}} \geq$$ $$\geq (1/2) (n/e)^n \sqrt{2 \pi n}.$$Re-write this as follows: $$((n^2-1)/4e^2)^{a(n-1)/2b} (\pi/2)^{a/b} (n^2-1)^{a/2b} ((n+1)/2e)^{a/b} \geq$$ $$\geq (1/2) (n/e)^n \sqrt{2 \pi n}.$$In other words: $$((n^2-1)/4e^2)^{an/2b} (\pi(n+1)/2)^{a/b} \geq (1/2) \sqrt{2 \pi n} (n/e)^n.$$Since $\pi(n+1)/2 \geq (1/2) \sqrt{2 \pi n}$ we are reduced to prove that $$((n^2-1)/4e^2)^{an/2b} \geq (n/e)^n,$$i.e. $$(n^2-1)^a \geq (n/e)^{2b} (4e^2)^a = 4^a e^{2(a-b)} n^{2b}.$$
\end{proof}

\begin{lemma} \label{pr2}
Let $n$ be an odd positive integer at least $5$, let $a$ be a $(2,n-2)$-cycle in $S_n$, and let $b$ be a $(n-1)$-cycle in $S_n$. No primitive maximal subgroup of $S_n$ contains $a$, no imprimitive maximal subgroup of $S_n$ contains $b$, and no intransitive maximal subgroup of $S_n$ contains both $a$ and $b$.
\end{lemma}

\begin{proof}
The second and the third statement are clear. If a primitive subgroup of $S_n$ contains $a$ then it contains the transposition $a^{n-2}$, thus it contains $A_n$ by the Jordan theory (cfr. for example \cite{cameron}, Theorem 6.15 and Exercise 6.6).
\end{proof}

In the rest of this section we will use the notations which we fixed in the introduction.

\begin{lemma} \label{ng}
Let $1 \leq k < 2m$ be an integer coprime to $2m$. In the following let the subscripts be identified with their reductions modulo $m$, and let $b_1:=1$, $b_2$, $\ldots$, $b_m \in S_n$, $x_1$, $\ldots$, $x_m \in A_n$. Let $M$ be a subgroup of $A_n$. The following holds.
\begin{enumerate}
\item Suppose $k<m$. For $d \in \{1,\ldots,m\}$ define $\tau_d$ to be $\tau$ if $d>m-k$, and $1$ if $d \leq m-k$. Then the element $(x_1,\ldots,x_m) \gamma^k \in G$ belongs to $N_G(M \times M^{b_2} \times \cdots \times M^{b_m})$ if and only if $$\eta_d := b_d x_d \tau_d b_{d+k}^{-1} \in N_{S_n}(M), \hspace{1cm} \forall d=1,\ldots,m.$$Moreover in this case $$\eta := \eta_1 \eta_{1+k} \eta_{1+2k} \cdots \eta_{1+(m-1)k} =$$ $$= x_1 \tau_1 x_{1+k} \tau_{1+k} \cdots x_{1+(m-1)k} \tau_{1+(m-1)k} \in N_{S_n}(M) - A_n.$$
\item Suppose $k>m$. For $d \in \{1,\ldots,m\}$ define $\tau_d$ to be $\tau$ if $d \leq 2m-k$, and $1$ if $d > 2m-k$. The element $(x_1,\ldots,x_m) \gamma^k \in G$ belongs to $N_G(M \times M^{b_2} \times \cdots \times M^{b_m})$ if and only if $$\eta_d := b_d x_d \tau_d b_{d+k-m}^{-1} \in N_{S_n}(M), \hspace{1cm} \forall d=1,\ldots,m.$$Moreover in this case $$\eta := \eta_1 \eta_{1+k-m} \eta_{1+2(k-m)} \cdots \eta_{1+(m-1)(k-m)} =$$ $$=x_1 \tau_1 x_{1+k-m} \tau_{1+k-m} \cdots x_{1+(m-1)(k-m)} \tau_{1+(m-1)(k-m)} \in N_{S_n}(M) - A_n.$$
\item If $N_{S_n}(M)$ contains $\eta$ (which depends only on $x_1,\ldots,x_m$), then there exist $a_2,\ldots,a_m \in A_n$ such that $$(x_1,\ldots,x_m) \gamma^k \in N_G(M \times M^{a_2} \times \cdots \times M^{a_m}).$$
\end{enumerate}
\end{lemma}

\begin{proof}
Assume first that $k<m$. The element $$(x_1,\ldots,x_m)\gamma^k = (x_1,\ldots,x_{m-k},x_{m-k+1}\tau,\ldots,x_m\tau)\delta^k$$ belongs to $N_G(M \times M^{b_2} \times \cdots \times M^{b_m})$ if and only if $$(M^{x_1} \times M^{b_2x_2} \times \cdots \times M^{b_{m-k}x_{m-k}} \times M^{b_{m-k+1}x_{m-k+1}\tau} \times \cdots \times M^{b_m x_m \tau})^{\delta^k} =$$ $$=M \times M^{b_2} \times \cdots \times M^{b_m},$$if and only if $$M^{b_{m-k+1}x_{m-k+1}\tau} \times \cdots \times M^{b_m x_m \tau} \times M^{x_1} \times M^{b_2x_2} \times \cdots \times M^{b_{m-k}x_{m-k}}=$$ $$= M \times M^{b_2} \times \cdots \times M^{b_m}.$$In other words: $$b_{m-k+1} x_{m-k+1} \tau, b_{m-k+2} x_{m-k+2} \tau b_2^{-1},\ldots, b_m x_m \tau b_k^{-1},$$ $$x_1 b_{k+1}^{-1}, b_2 x_2 b_{k+2}^{-1},\ldots, b_{m-k} x_{m-k} b_m^{-1} \in N_{S_n}(M).$$For $d \in \{1,\ldots,m\}$ define $\tau_d$ to be $\tau$ if $d>m-k$, and $1$ if $d \leq m-k$. The conditions we have are the following: $$\eta_d := b_d x_d \tau_d b_{d+k}^{-1} \in N_{S_n}(M), \hspace{1cm} d=1,\ldots,m.$$Observe that since $k$ and $m$ are coprime, $$\{\tau_1, \tau_{1+k}, \tau_{1+2k},\ldots, \tau_{1+(m-1)k}\} = \{\tau_1,\ldots,\tau_m\}.$$Now $$\eta := \eta_1 \eta_{1+k} \eta_{1+2k} \cdots \eta_{1+(m-1)k} =$$ $$= x_1 \tau_1 x_{1+k} \tau_{1+k} \cdots x_{1+(m-1)k} \tau_{1+(m-1)k} \in N_{S_n}(M)$$is an odd element of $S_n$ since $\eta \equiv \tau^k \mod(A_n)$ and $k$ is odd (being coprime to $2m$).

Point (2) follows easily from point (1) by noticing that $((x_1,\ldots,x_m) \gamma^k)^{-1} = (x_1^{-1},\ldots,x_m^{-1})^{\gamma^k} \gamma^{2m-k}$.

Let us prove point (3). Suppose that the normalizer of $M$ in $S_n$ contains $\eta$. Assume that $k<m$ (the case $k > m$ is similar). For fixed elements $b_2,\ldots,b_m \in S_n$ define $\eta_d := b_d x_d \tau_d b_{d+k}^{-1}$, for $d=1,\ldots,m$, and now choose $b_2,\ldots,b_m$ in such a way that $\eta_{1+k}$, $\eta_{1+2k}$,\ldots, $\eta_{1+(m-1)k} \in N_{S_n}(M)$. Let $\eta_1$ be the element of $S_n$ such that $\eta_1 \eta_{1+k} \cdots \eta_{1+(m-1)k} = \eta$. Then since $\eta \in N_{S_n}(M)$, also $\eta_1 \in N_{S_n}(M)$. Now, a suitable power of $(x_1,\ldots,x_m)\gamma^k$ is of the form $(y_1,\ldots,y_m) \gamma$, with $y_1,\ldots,y_m \in A_n$. Since the element $(y_1,\ldots,y_m) \gamma \in G$ belongs to $N_G(M \times M^{b_2} \times \cdots \times M^{b_m})$ we have $$b_m y_m \tau, y_1 b_2^{-1}, b_2 y_2 b_3^{-1},\ldots, b_{m-1}y_{m-1}b_m^{-1} \in N_{S_n}(M).$$We may choose $a_2 := y_1$, $a_3 := y_1 y_2$,\ldots, $a_m := y_1 y_2 \cdots y_{m-1}$. In this way we get $M^{b_i}=M^{a_i}$ and $a_i \in A_n$, for $i=2,\ldots,m$.
\end{proof}

From the proof of this proposition it easily follows that:

\begin{cor} \label{-an}
If $M \leq A_n$, $b_2,\ldots,b_m \in \Aut(A_n)$ and $N_G(M \times M^{b_2} \times \cdots \times M^{b_m})$ contains an element of the form $(x_1,\ldots,x_m) \gamma$ with $x_1,\ldots,x_m \in A_n$ then there exist $a_2,\ldots,a_m \in A_n$ such that $M^{b_i} = M^{a_i}$ for $i=2,\ldots,m$.
\end{cor}

\begin{lemma} \label{pr1}
Let $r$ be a divisor of $m$, and let $x_1,\ldots,x_m \in A_n$, $a_1:=1,a_2,\ldots,a_m \in S_n$. Let $M$ be a subgroup of $A_n$. The element $(x_1,\ldots,x_m) \gamma^r \in G$ belongs to $N_G(M \times M^{a_2} \times \cdots \times M^{a_m})$ if and only if the following conditions are satisfied: $$a_{m-r+i}x_{m-r+i} \tau a_i^{-1} \in N_{S_n}(M) \hspace{1cm} \forall i=1,\ldots,r;$$ $$a_i x_i a_{r+i}^{-1} \in N_{S_n}(M) \hspace{1cm} \forall i=1,\ldots,m-r.$$In particular $$x_i x_{i+r} x_{i+2r} \cdots x_{i+m-r} \tau \in N_{S_n}(M)^{a_i} \hspace{1cm} \forall i=1,\ldots,r.$$Now assume that $m$ is odd. Then the element $(x_1,\ldots,x_m)\gamma^2 \in G$ belongs to $N_G(M \times M^{a_2} \times \cdots \times M^{a_m})$ if and only if the following conditions are satisfied: $$a_{m-1}x_{m-1}\tau, a_mx_m\tau a_2^{-1},$$ $$x_1a_3^{-1}, a_2x_2a_4^{-1},\ldots, a_{m-2}x_{m-2}a_m^{-1} \in N_{S_n}(M).$$In particular $$x_1x_3 \cdots x_m\tau x_2x_4 \cdots x_{m-1}\tau \in N_{S_n}(M).$$
\end{lemma}

\begin{proof}
The element $(x_1,\ldots,x_m) \gamma^r = (x_1,\ldots,x_{m-r},x_{m-r+1}\tau,\ldots,x_m\tau) \delta^r$ normalizes $M \times M^{a_2} \times \cdots \times M^{a_m}$ if and only if $$(M \times M^{a_2} \times \cdots \times M^{a_m})^{(x_1,\ldots,x_{m-r},x_{m-r+1}\tau,\ldots,x_m\tau)\delta^r} = M \times M^{a_2} \times \cdots \times M^{a_m},$$in other words $$M^{a_{m-r+1}x_{m-r+1} \tau} \times \cdots \times M^{a_m x_m \tau} \times M^{x_1} \times M^{a_2x_2} \times \cdots \times M^{a_{m-r} x_{m-r}} =$$ $$= M \times M^{a_2} \times \cdots \times M^{a_m},$$and this leads to what is stated.

Now assume $m$ is odd. The element $$(x_1,\ldots,x_m) \gamma^2 = (x_1,\ldots,x_{m-2},x_{m-1}\tau,x_m\tau) \delta^2$$ normalizes $M \times M^{a_2} \times \cdots \times M^{a_m}$ if and only if $$(M \times M^{a_2} \times \cdots \times M^{a_m})^{(x_1,\ldots,x_{m-2},x_{m-1}\tau,x_m\tau)\delta^2} = M \times M^{a_2} \times \cdots \times M^{a_m},$$in other words $$M^{a_{m-1}x_{m-1} \tau} \times M^{a_m x_m \tau} \times M^{x_1} \times M^{a_2x_2} \times \cdots \times M^{a_{m-2} x_{m-2}} =$$ $$= M \times M^{a_2} \times \cdots \times M^{a_m},$$and this leads to what is stated.
\end{proof}

\begin{lemma} \label{m/q}
Let $r$ be a divisor of $m$. The element $(x_1,\ldots,x_m)\gamma$ normalizes $$\Delta := \{(y_1,\ldots, y_{m/r}, y_1^{b_{21}},\ldots, y_{m/r}^{b_{2,m/r}},\ldots, y_1^{b_{r,1}},\ldots, y_{m/r}^{b_{r,m/r}})\ |\ y_1,\ldots,y_{m/r} \in A_n\}$$if and only if (here $b_{1i}=1$ for all $i=1,\ldots,m/r$) $$b_{r,m/r} x_m \tau b_{i1} = b_{i-1,m/r} x_{(i-1)m/r} \hspace{1cm} \forall i=2,\ldots,r$$and $$x_j b_{i,j+1} = b_{i,j} x_{(i-1)(m/r)+j} \hspace{1cm} \forall i=2,\ldots,r,\ j=1,\ldots,m/r-1.$$In particular $$x_1 \cdots x_m \tau = [x_1 \cdots x_{m/r-1}(b_{r,m/r} x_m \tau)]^r.$$For $b \in S_n$ let $l_r(b)$ be the number of elements $s \in S_n$ such that $s^r=b$. Then $$|\{(x_1,\ldots,x_m) \gamma \in N_G(\Delta)\ |\ x_1 \cdots x_m\tau=b\}| = l_r(b) \cdot |A_n|^{m/r-1}.$$In particular this number is $0$ if $b \in A_n$ or if $r$ is even.
\end{lemma}

\begin{proof}
It is a direct computation. The element $(x_1,\ldots,x_m) \gamma$ belongs to $N_G(\Delta)$ if and only if for every $y_1,\ldots,y_{m/r} \in A_n$ the element $$(y_{m/r}^{b_{r,m/r}x_m\tau},y_1^{x_1},\ldots,y_{m/r}^{x_{m/r}},y_1^{b_{21}x_{m/r+1}},\ldots,y_{m/r}^{b_{2,m/r}x_{2m/r}},\ldots,$$ $$y_1^{b_{r,1}x_{(r-1)m/r+1}},\ldots,y_{m/r-1}^{b_{r,m/r-1}x_{m-1}})$$belongs to $\Delta$, and this leads to the stated conditions.

Using these conditions we see that for every $1 \leq i \leq r-1$, $$x_1 \cdots x_{m/r-1} b_{r,m/r} x_m \tau =$$ $$= b_{i,1} x_{(i-1)m/r+1} x_{(i-1)m/r+2} \cdots x_{(i-1)m/r+m/r-1} x_{im/r} b_{i+1,1}^{-1},$$and $$x_1 \cdots x_{m/r-1} b_{r,m/r} x_m \tau = b_{r,1} x_{(r-1)m/r+1} \cdots x_{m-1} x_m \tau.$$It follows that $$(x_1 \cdots x_{m/r-1} b_{r,m/r} x_m \tau)^r = x_1 \cdots x_m \tau.$$The last two statements follow easily from the first two.
\end{proof}

\section{Proof of Theorem \ref{fin}}

In this section we prove Theorem \ref{fin} for $m \geq 2$ (the case $m=1$ is proved in \cite{maroti2}).

The next definition was introduced in \cite{maroti2}.

\begin{defi}[Definite unbeatability]
\label{d1} Let $X$ be a finite group. Let $\mathcal{H}$ be a set
of proper subgroups of $X$, and let $\Pi \subseteq X$. Suppose
that the following four conditions hold on $\mathcal{H}$ and
$\Pi$.
\begin{enumerate}
\item $\Pi \cap H \neq \emptyset$ for every $H \in \mathcal{H}$;

\item $\Pi \subseteq \bigcup_{H \in \mathcal{H}} H$;

\item $\Pi \cap H_{1} \cap H_{2} = \emptyset$ for every distinct
pair of subgroups $H_{1}$ and $H_{2}$ of $\mathcal{H}$;

\item $|\Pi \cap K| \leq |\Pi \cap H|$ for every $H \in
\mathcal{H}$ and $K < X$ with $K \not \in \mathcal{H}$.
\end{enumerate}
Then $\mathcal{H}$ is said to be definitely unbeatable on $\Pi$.
\end{defi}

For $\Pi \subseteq X$ let $\sigma_X(\Pi)$ be the least cardinality
of a family of proper subgroups of $X$ whose union contains $\Pi$.
The next lemma is straightforward so we state it without proof.

\begin{lemma}
\label{l6} If $\mathcal{H}$ is definitely unbeatable on $\Pi$ then
$\sigma_X(\Pi)=|\mathcal{H}|$.
\end{lemma}

It follows that if $\mathcal{H}$ is definitely unbeatable on $\Pi$
then $|\mathcal{H}| = \sigma_X(\Pi) \leq \sigma(X)$.

Let us fix the notations.

\begin{nota} \label{nota}
Let $n,m$ be positive integers, with $m \geq 2$ and $n \geq 5$. Let $A,B$ be two fixed subsets of $S_n-A_n$, and let $C$ be a fixed subset of $A_n$. For a prime divisor $r$ of $m$ define $\Omega_r$ to be the set $$\{(x_1,\ldots,x_m) \gamma^r\ |\ x_1x_{1+r}x_{1+2r} \cdots x_{1+m-r} \tau \in A,\ x_2x_{2+r}x_{2+2r} \cdots x_{2+m-r} \tau \in B\}.$$If $m$ is odd let $$\Omega_2 := \{(x_1,\ldots,x_m)\gamma^2\ |\ x_1x_3 \cdots x_m\tau x_2x_4 \cdots x_{m-1}\tau \in C\}.$$For a prime divisor $r$ of $2m$ let $H_r$ be the pre-image of $\langle \gamma^r \rangle$ via the projection $G \to \langle \gamma \rangle$. Let $\Pi$ be a fixed subset of $S_n-A_n$, and let $$\Omega_1 := \{(x_1,\ldots,x_m) \gamma\ |\ x_1 \cdots x_m \tau \in \Pi\}.$$
\end{nota}

Assume that $n \geq 5$ is odd. Let $K_1,\ldots,K_t$ be the intransitive maximal subgroups of $A_n$. Let $\Sigma$ be the subset of $S_n$ consisting of the $(k,n-k)$-cycles where $1 \leq k \leq n-1$, and let $\Pi$ be a fixed subset of $\Sigma$. Call $I := \{i \in \{1,\ldots,t\}\ |\ N_{S_n}(K_i) \cap \Pi \neq \emptyset\}$. Let $$\mathcal{L} := \{N_G(K_i \times K_i^{a_2} \times \cdots \times K_i^{a_m})\ |\ i \in I,\ a_2,\ldots,a_m \in A_n\}.$$Let $A$ be the set of the $(2,n-2)$-cycles of $S_n$, let $B$ be the set of the $(n-1)$-cycles of $S_n$, for $m$ odd let $C$ be:

\begin{itemize}
\item the set of the $n$-cycles of $S_n$ if either $n \geq 7$, or $n=5$ and $m \not \in \{5,7\}$;
\item a subset of $S_5$ consisting of $12$ $5$-cycles, two in each Sylow $5$-subgroup, if $n=5$ and $m \in \{5,7\}$.
\end{itemize}

If $m$ is even or $(n,m)=(5,3)$ let $C=\emptyset$. We have $|A|=|A_n|/(n-2)$, $|B|=2|A_n|/(n-1)$, $|C|=2|A_n|/n$ if $n \geq 7$ or $n=5$, $m \not \in \{3,5,7\}$, $|C|=12$ if $n=5$, $m \in \{5,7\}$, and $|\Omega_r|=\frac{2}{(n-1)(n-2)} |A_n|^m$ if $r \neq 2$ or $m$ is even, while if $r=2$ and $m$ is odd then $|\Omega_r|=(2/n)|A_n|^m$. Suppose we are in one of the following cases:

\begin{enumerate}
\item $n=5$ and $\Pi = \{(2354), (4521), (4132),(1253), (4531), (3245), (1352), \\ (2314), (4125), (3541)\}$;
\item $n \geq 7$ and $\Pi = \Sigma$.
\end{enumerate}

Let $r_1,\ldots,r_{\omega(2m)}$ be the distinct prime factors of $2m$. By Lemma \ref{ng} and Corollary \ref{-an} the family $\mathcal{H} := \mathcal{L} \cup \{ H_{r_1},\ldots, H_{r_{\omega(2m)}} \}$ covers $G$ if $n \neq 5$. In fact the odd elements of $S_n$ are covered by the intransitive maximal subgroups of $S_n$.

\begin{prop} \label{duprm}
With the notations and assumptions above, we have:

\begin{enumerate}
\item If $(n,m) \neq (5,3)$, $\mathcal{H}$ is definitely unbeatable on $\Omega := \Omega_1 \cup \Omega_{r_1} \cup \cdots \cup \Omega_{r_{\omega(2m)}}$.
\item $\mathcal{L}$ is definitely unbeatable on $\Omega_1$.
\end{enumerate}

\end{prop}

\begin{proof}

We will verify the four conditions of Definition \ref{d1} for both $\mathcal{H}$ and $\mathcal{L}$. Lemmas \ref{pr2} and \ref{pr1} imply that if $H$ is a maximal subgroup of $G$ of product type and $r$ is a prime divisor of $m$ then $H \cap \Omega_r = \emptyset$; in particular $H \cap \Omega = H \cap (\Omega_1 \cup \Omega_2)$. If $H \in \mathcal{L}$ then $H \cap \Omega_2 = \emptyset$. Moreover $\Omega_r \subset H_r$ for every prime divisor $r$ of $2m$ and $H_r \cap H_s \cap \Omega = \emptyset$ for every two distinct prime divisors $r,s$ of $2m$. All this implies that the first three conditions of Definition \ref{d1} hold for $\mathcal{H}$ if they hold for $\mathcal{L}$. We will check them now.

Recall first that if $K$ is a subgroup of $A_n$ and $x_1,\ldots,x_m,a_2,\ldots,a_m \in A_n$ then \\ $(x_1,\ldots,x_m) \gamma \in N_G(K \times K^{a_2} \times \cdots \times K^{a_m})$ if and only if $$a_m x_m \tau \in N_{S_n}(K),\ x_1 \in Ka_2,\ x_2 \in a_2^{-1} K a_3,\ \cdots ,\ x_{m-1} \in a_{m-1}^{-1} K a_m.$$

\begin{enumerate}

\item We show that $\Omega_1 \cap H \neq \emptyset$ for every $H = N_G(K_i \times K_i^{a_2} \times \cdots \times K_i^{a_m}) \in \mathcal{L}$. Choose the element $(x_1,\ldots,x_m)\gamma$ in the intersection in this way: $x_1=a_2$, $x_2=a_2^{-1}a_3$,\ldots, $x_{m-1}=a_{m-1}^{-1}a_m$ and $x_m$ such that $x_1 \cdots x_m \tau \in \Pi \cap N_{S_n}(K_i)$.

\item We show that $\Omega_1 \subseteq \bigcup_{H \in \mathcal{L}} H$. Given $(x_1,\ldots,x_m)\gamma \in \Omega_1$ choose $i \in I$ such that $x_1 \cdots x_m \tau \in N_{S_n}(K_i)$ and $a_2=x_1$, $a_3=x_1x_2$,\ldots, $a_m=x_1x_2 \cdots x_{m-1}$. Then choose $H := N_G(K_i \times K_i^{a_2} \times \cdots \times K_i^{a_m})$.

\item We show that $\Omega_1 \cap N_G(K_i \times K_i^{a_2} \times \cdots \times K_i^{a_m}) \cap N_G(K_j \times K_j^{b_2} \times \cdots \times K_j^{b_m}) = \emptyset$ for $N_G(K_i \times K_i^{a_2} \times \cdots \times K_i^{a_m}) \neq N_G(K_j \times K_j^{b_2} \times \cdots \times K_j^{b_m})$ belonging to $\mathcal{L}$. If $(x_1,\ldots,x_m)\gamma$ belongs to the stated intersection then $x_1 \cdots x_m \tau \in N_{S_n}(K_i) \cap N_{S_n}(K_j) \cap \Pi$ with $i \neq j$ (which is impossible) or $i=j$ and $$x_k \in a_k^{-1}K_ia_{k+1} \cap b_k^{-1} K_i b_{k+1}$$for $k=1,\ldots,m$, where $a_1:=1$. This easily implies that $K_i^{a_k}=K_i^{b_k}$ for $k=2,\ldots,m$, contradiction.

\end{enumerate}

We now prove that $|H \cap \Omega| \geq |H' \cap \Omega|$ for every $H \in \mathcal{H}$, $H'$ maximal subgroup of $G$ with $H' \not \in \mathcal{H}$. Note that this indeed proves condition (4) of Definition \ref{d1} for both $\mathcal{H}$ and $\mathcal{L}$ since for every prime divisor $r$ of $2m$ and every $H \in \mathcal{L}$ we have $H_r \cap \Omega_1 = \emptyset$ and $H \cap \Omega_r = \emptyset$.

First we prove that if $K \not \in \{K_i\ |\ i \in I\}$ is a subgroup of $A_n$ of the form $R \cap A_n$ where $R$ is a maximal subgroup of $S_n$ (cfr. section \ref{intro}) then $$|\Omega \cap N_G(K_i \times K_i^{a_2} \times \cdots \times K_i^{a_m})| \geq |\Omega \cap N_G(K \times K^{b_2} \times \cdots \times K^{b_m})|.$$ Notice that since the right hand side of this inequality is zero if $K$ is intransitive (this can happen if $n=5$), we may assume that $K$ is transitive. As we have already noticed this inequality re-writes as $$|\Omega_1 \cap N_G(K_i \times K_i^{a_2} \times \cdots \times K_i^{a_m})| \geq |(\Omega_1 \cup \Omega_2) \cap N_G(K \times K^{b_2} \times \cdots \times K^{b_m})|.$$ The size of $\Omega_1 \cap N_G(L \times L^{a_2} \times \cdots \times L^{a_m})$ in general (for a subgroup $L$ of $A_n$ and some $a_2,\ldots,a_m \in S_n$) is $|L|^{m-1} \cdot |N_{S_n}(L) \cap \Pi|$, and if $m$ is odd the size of $\Omega_2 \cap N_G(L \times L^{a_2} \times \cdots \times L^{a_m})$ is $|L|^{m-1} \cdot |L \cap C|$. Therefore we have to show that $$|K_i|^{m-1} \cdot |N_{S_n}(K_i) \cap \Pi| \geq |K|^{m-1} \cdot |N_{S_n}(K) \cap (\Pi \cup C)|. \hspace{.4cm} (\ast)$$

\begin{itemize}
\item Suppose $n=5$. The transitive maximal subgroups of $A_5$ have order $10$. Moreover the only intransitive maximal subgroups of $A_5$ whose normalizers in $S_5$ intersect $\Pi$ are the five point stabilizers. If $m$ is even or $m=3$ then $C=\emptyset$ and $|N_{S_5}(K_i) \cap \Pi| = |N_{S_5}(K) \cap \Pi| = 2$ for every $i \in I$, thus $(\ast)$ is true. If $m \not \in \{5,7\}$ is odd then $|N_{S_n}(K) \cap (\Pi \cup C)| = 6$ and $(\ast)$ becomes $12^{m-1} \cdot 2 \geq 10^{m-1} \cdot 6$, which is true for $m \geq 8$. If $m \in \{5,7\}$ then $|N_{S_n}(K) \cap (\Pi \cup C)| = 4$ and $(\ast)$ becomes $12^{m-1} \cdot 2 \geq 10^{m-1} \cdot 4$, which is true.
\item Suppose $n=7$. The left hand side is at least $72^{m-1} \cdot 12$. Since the transitive maximal subgroups of $S_7$ different from $A_7$ have size $42$ and contain $20$ elements of $\Pi \cup C$, it suffices to show that $72^{m-1} \cdot 12 \geq 21^{m-1} \cdot 20$, i.e. $(72/21)^m \geq 40/7$, which is true for $m \geq 2$.
\item Suppose $n=9$. The smallest maximal intransitive subgroup of $A_9$ is the one of type $(4,5)$, it has size $1440$ and the size of the intersection of its normalizer in $S_9$ with $\Pi$ is the smallest possible, $3! \cdot 4! = 144$. Thus the left hand side of $(\ast)$ is at least $1440^{m-1} \cdot 144$. The right hand side is at most $\max (216^{m-1} \cdot 72, 648^{m-1} \cdot 432)$ (note that the maximal subgroups of $A_9$ isomorphic to $\text{Aut}(PSL(2,8))$ are not of the form $R \cap A_9$ with $R$ maximal in $S_9$: cfr. section \ref{intro}). Therefore it suffices to show that $1440^{m-1} \cdot 144 \geq 648^{m-1} \cdot 432$, and this is true for $m \geq 3$. If $m=2$ then $C=\emptyset$ and it suffices to show that $1440 \cdot 144 \geq 648 \cdot 288$ (recall that the imprimitive maximal subgroups of $S_9$ contain $144$ $9$-cycles and $288$ $(6,3)$-cycles), which is true.
\item Suppose $n \geq 11$. Then $|K_i| \geq |K|$ by Corollary \ref{corsizes}, and the inequality $|N_{S_n}(K_i) \cap \Pi| \geq |N_{S_n}(K) \cap (\Pi \cup C)|$ is proved in claim 3.2 of \cite{maroti2}.
\end{itemize}

Now we prove that if $N_G(\Delta)$ is a maximal subgroup of $G$ of diagonal type (its existence implies that $m$ is not a power of $2$ by Lemma \ref{m/q}) and $i \in \{1,\ldots,t\}$, $a_2,\ldots,a_m \in A_n$ then $$|\Omega \cap N_G(K_i \times K_i^{a_2} \times \cdots \times K_i^{a_m})| \geq |\Omega \cap N_G(\Delta)|.$$The right hand side is at most $|N_G(\Delta)| \leq 2m |A_n|^{m/p}$, where $p$ is the smallest prime divisor of $m$, hence we are reduced to prove that $|K_i|^{m-1} \cdot |N_{S_n}(K_i) \cap \Pi| \geq 2m|A_n|^{m/p}$. Since if $K_i$ is of type $(k,n-k)$ then $|N_{S_n}(K_i) \cap \Pi|=(k-1)!(n-k-1)!$, we obtain $(2/(k(n-k)))|K_i|^m \geq 2m|A_n|^{m/p}$. Since $k(n-k) \leq ((n-1)/2)((n+1)/2)$, it suffices to show that $$\frac{8}{n^2-1} |K_i|^m \geq 2m \cdot |A_n|^{m/p}. \hspace{2cm} (1)$$Note that if $s$ is a divisor of $m$ and $L_s$ denotes the set of elements of $G$ of the form $(x_1,\ldots,x_m) \gamma^s$ then $|N_G(\Delta) \cap L_s|=|\Delta|$. Therefore by Lemma \ref{m/q} if $N_G(\Delta)$ is of diagonal type $2$ then it suffices to show that $$\frac{8}{n^2-1} |K_i|^m \geq \omega(m) \cdot |A_n|^{m/2}. \hspace{2cm} (2)$$

\begin{itemize}

\item If $n=5$ then $|K_i| \geq \frac{1}{2} 2! 3! = 6$ and (1) is true for $m \geq 3$. If $n=9$ then $|K_i| \geq \frac{1}{2} 4! 5! = 1440$ and (1) is true for $m \geq 3$. If $n=11$ then $|K_i| \geq \frac{1}{2} 5! 6! = 43200$ and (1) is true for $m \geq 2$. If $n=13$ then $|K_i| \geq \frac{1}{2} 6! 7! = 1814400$ and (1) is true for $m \geq 2$.

\item Suppose $n=7$. Then $|K_i| \geq 72$, thus it suffices to show (1): $72^m \geq 12m \cdot 2520^{m/2}$, i.e. $(72/\sqrt{2520})^m/m \geq 12$. This is true for $m \geq 15$. If $p \geq 3$ it suffices to show that $(72/\sqrt[3]{2520})^m/m \geq 12$, which is true for $m \geq 3$. Thus we are done if $p$ is odd. If $m \in \{10,12,14\}$ then $\omega(m)=2$ and using (2) we are reduced to show that $(72/\sqrt{2520})^m \geq 12$, which is true.

We are left with the case $m=6$. It is easy to see that in general if $H$ is a maximal subgroup of $G$ of diagonal type $2$ and $r$ is a prime divisor of $m$ then $|\Omega_r \cap H| \leq |A_n|^{m/2-1} \cdot \min(|A|,|B|)$ (just use the definition of $\Omega_r$). In our case $\min(|A|,|B|)=|A|=504$, and $72^6 \geq \omega(6) \cdot 6 \cdot 2520^2 \cdot 504$.

\item Suppose $n \geq 15$. Then $|K_i|^{3/2} \geq |A_n|$ by Lemma \ref{tremezz}, so using (1) we are reduced to prove that $(8/(n^2-1))|A_n|^{\frac{2}{3}m} \geq 2m|A_n|^{m/2}$, i.e. $|A_n|^{m/6} \geq (m/4) (n^2-1)$. This is clearly true for every $m$ since $n \geq 15$.

\end{itemize}

Now we prove that if $(n,m) \neq (5,3)$ then $|H_r \cap \Omega| \geq |H \cap \Omega|$ for every maximal subgroup $H$ of $G$ of product type out of $\mathcal{H}$ and for every prime divisor $r$ of $2m$. Let $L$ be the transitive subgroup of $A_n$ such that $H=N_G(L \times L^{a_2} \times \cdots \times L^{a_m})$. Note that $$|H \cap \Omega| = |H \cap (\Omega_1 \cup \Omega_2)| = |L|^{m-1} \cdot (|N_{S_n}(L) \cap \Pi|+|L \cap C|).$$ Suppose first that $r \neq 2$ or $m$ is even. All we have to prove is that $$\frac{2}{(n-1)(n-2)} |A_n|^m = |\Omega_r| = |H_r \cap \Omega| \geq |H \cap (\Omega_1 \cup \Omega_2)| =$$ $$= |L|^{m-1} \cdot (|N_{S_n}(L) \cap \Pi|+|L \cap C|).$$This is easily seen to be true for $n \in \{5,7,9\}$. Suppose $n \geq 11$. It suffices to show that $\frac{2}{(n-1)(n-2)} |A_n|^m \geq 2|R|^m$ for any maximal transitive subgroup $R$ of $S_n$ different from $A_n$, i.e. $(|S_n:R|/2)^m \geq (n-1)(n-2)$, and this is true by Corollary \ref{corsizes}, being true for $m=1$: $|S_n:R|/2 \geq \binom{n}{5}/2 > (n-1)(n-2)$ since $n>8$.

Assume now that $r=2$ and $m$ is odd. All we have to prove is that $$|C| \cdot |A_n|^{m-1} = |\Omega_2| = |H_2 \cap \Omega| \geq |H \cap (\Omega_1 \cup \Omega_2)| =$$ $$= |L|^{m-1} \cdot (|N_{S_n}(L) \cap \Pi|+|L \cap C|).$$It suffices to prove that for every transitive subgroup $R$ of $S_n$ not containing $A_n$ we have $|C| \cdot |A_n|^{m-1} \geq 2|R|^m$, i.e. $(|S_n:R|/2)^m \geq |S_n|/|C|$. If $n > 5$ this follows from $|S_n:R| \geq n$, if $n=5$ this follows from $|C| \geq 12$.

Now we prove that if $(n,m) \neq (5,3)$ then $|H_r \cap \Omega| \geq |H \cap \Omega|$ for every prime divisor $r$ of $2m$ and every maximal subgroup $H$ of $G$ of diagonal type. Notice that $|H| \leq 2m|A_n|^{m/2}$, hence if $r \neq 2$ or $m$ is even we are reduced to prove that $2|A_n|^m/((n-1)(n-2)) \geq 2m |A_n|^{m/2}$, and this is clearly true for every $m$ and $n \geq 5$. If $r=2$ and $m$ is odd we have to prove that $(2/n) |A_n|^m \geq 2m |A_n|^{m/2}$, and this is clearly true for every $m$ and $n \geq 5$.
\end{proof}

Note that Proposition \ref{duprm} implies Theorem \ref{fin} if $n > 5$ is odd.

\begin{obs} \label{obs1}
Let $\mathcal{K}$ be a minimal cover of the finite group $X$, so that $|\mathcal{K}|=\sigma(X)$, and let $\mathcal{K}_1$ be a subset of $\mathcal{K}$. Let $\Omega$ be a subset of $X-\bigcup_{K \in \mathcal{K}_1} K$. Then $|\mathcal{K}_1|+\sigma_X(\Omega) \leq \sigma(X)$, where $\sigma_X(\Omega)$ denotes the least number of proper subgroups of $X$ needed to cover $\Omega$.
\end{obs}

Suppose that $n=5$ and all the prime divisors of $m$ belong to $\{2,3\}$. Fix a minimal cover $\mathcal{K}$ of $G$. Let $\mathcal{K}_0$ be the family of the maximal subgroups of $G$ of the form $N_G(M \times M^{a_2} \times \cdots \times M^{a_m})$ with $a_2,\ldots,a_m \in A_5$ and $M$ an intransitive maximal subgroup of $A_5$ of type $(3,2)$. Since the $(3,2)$-cycles are not of the form $x^2$ or $x^3$ for $x \in S_5$, by Lemma \ref{m/q} the only maximal subgroups of $G$ which contain elements of the form $(x_1,\ldots,x_m) \gamma$ where $x_1 \cdots x_m \tau$ is a $(3,2)$-cycle are the subgroups in $\mathcal{K}_0$. In particular $\mathcal{K}_0 \subset \mathcal{K}$. In the following we use Notations \ref{nota}, with $A$ the set of the $(3,2)$-cycles, $B$ the set of the $4$-cycles and $C$ the set of the $5$-cycles.

Suppose that $m$ is even, and let $\mathcal{K}_1 := \mathcal{K}_0$. For every $K \in \mathcal{K}_1$ we have $\Omega_1 \cap K = \Omega_2 \cap K = \Omega_3 \cap K = \emptyset$, thus by Observation \ref{obs1} and Proposition \ref{duprm} $|\mathcal{K}_1|+|\mathcal{H}| \leq \sigma(G)$, and we have equality since $\mathcal{K}_1 \cup \mathcal{H}$ covers $G$.

Suppose that $m$ is a power of $3$, and let $\mathcal{K}_1 := \mathcal{K}_0 \cup \{H_2,H_3\}$. If either $H_2 \not \in \mathcal{K}$ or $H_3 \not \in \mathcal{K}$ then in order to cover $\Omega_2 \cup \Omega_3$ we need at least $$\frac{\min \{|\Omega_2|, |\Omega_3|\}}{|N_G(\Delta)|} = \frac{|A| \cdot |B| \cdot |A_5|^{m-1}}{2m \cdot |A_5|^{m/3}} = (5/m) \cdot 60^{2m/3-1}$$subgroups, where $N_G(\Delta)$ is a maximal subgroup of $G$ of diagonal type. Since $\sigma(G) \leq 2 + 5^m + 10^m$, we obtain that $10^m + (5/m) 60^{2m/3-1} \leq 2 + 5^m + 10^m$, contradiction. Therefore $\mathcal{K}_1 \subset \mathcal{K}$. Since $\Omega_1 \cap K = \emptyset$ for every $K \in \mathcal{K}_1$, by Observation \ref{obs1} and Proposition \ref{duprm} we obtain that $2+5^m+10^m \leq \sigma(G)$, thus we have equality.

Assume now that $n$ is any positive integer at least $5$. The following observation follows easily from the proof of Proposition \ref{duprm}.

\begin{obs} \label{stim}
Let $\mathcal{A}$ be a family of proper subgroups of $A_n$, and let $$\mathcal{K} := \{N_G(M \times M^{a_2} \times \cdots \times M^{a_m})\ |\ a_2,\ldots,a_m \in A_n,\ M \in \mathcal{A}\}.$$Let $\Pi$ be a subset of $S_n$ such that $\mathcal{A}$ is definitely unbeatable on $\Pi$. Let $$\Omega := \{(x_1,\ldots,x_m) \gamma \in G\ |\ x_1 \cdots x_m \tau \in \Pi\}.$$Suppose that the following two conditions hold:

\begin{enumerate}
\item $|M| \geq |K|$ for every $M \in \mathcal{A}$ and every maximal subgroup $K$ of $A_n$ such that $N_{S_n}(K) \cap \Pi \neq \emptyset$.
\item $|M|^{m-1} \cdot |N_{S_n}(M) \cap \Pi| \geq |H \cap \Omega|$ for every $M \in \mathcal{A}$ and every maximal subgroup $H$ of $G$ of diagonal type. Note that this is true if $$|M|^{m-1} \cdot |N_{S_n}(M) \cap \Pi| \geq 2m |A_n|^{m/p},$$ where $p$ is the smallest prime divisor of $m$ such that there exists a maximal subgroup of $G$ of diagonal type $p$ whose intersection with $\Omega$ is non-empty.
\end{enumerate}

Then the family $\mathcal{K}$ of subgroups of $G$ is definitely unbeatable on $\Omega$. In particular $|\mathcal{K}| \leq \sigma(G)$.
\end{obs}

Let us apply this observation to the cases we are left with.

Let $n=5$. Let $\mathcal{A}$ be the set of the intransitive maximal subgroups of $A_5$ of type $(3,2)$ and let $\Pi$ be the set of the $(3,2)$-cycles in $S_5$. Condition (1) of Observation \ref{stim} is clearly verified. Let us prove condition (2). By Lemma \ref{m/q} we may assume $p \geq 5$ (the elements of $\Pi$ have no square roots nor cubic roots in $S_5$). The inequality $6^{m-1} \cdot 2 \geq 2m \cdot 60^{m/p}$ is then true. We obtain $\sigma(G) \geq 10^m$.

Let $n=6$. Fix a minimal cover $\mathcal{M}$ of $G$ consisting of maximal subgroups. Let $\mathcal{K}_0$ be the family of the maximal subgroups of $G$ of the form $N_G(M \times M^{a_2} \times \cdots \times M^{a_m})$ where $M$ is a subgroup of $A_6$ isomorphic to $A_5$, so that $|\mathcal{K}_0| = 12 \cdot 6^{m-1}$. Let us use Notations \ref{nota}. Let $\mathcal{K}_1$ be the set consisting of the subgroups in $\mathcal{K}_0$ and the subgroups $H_r$ for $r$ a prime divisor of $m$. Since $S_6-A_6$ is covered by the two conjugacy classes of maximal subgroups of $S_6$ isomorphic to $S_5$, $\mathcal{K}_1 \cup \{H_2\}$ covers $G$, in particular $\sigma(G) \leq \omega(2m) + 2 \cdot 6^m$. It is easy to see that $H \cong A_5^m \rtimes C_{2m}$ for every $H \in \mathcal{K}_0$, therefore $$\sigma(H) \geq 10^m > \omega(2m) + 2 \cdot 6^m \geq \sigma(G).$$By Lemma 1 in \cite{gar} we deduce that $\mathcal{K}_0 \subset \mathcal{M}$. Let $A$ be the set of the $(3,2)$-cycles in $S_6$, let $B$ be the set of the $6$-cycles in $S_6$, and let $C$ be the set of the $3$-cycles in $S_6$. Since no subgroup of $S_6$ intersects both $A$ and $B$, $H \cap \Omega_r = \emptyset$ for every prime divisor $r$ of $m$ and every maximal subgroup $H$ of $G$ of product type. If $H$ is a maximal subgroup of $G$ of diagonal type (in particular $m$ is not a power of $2$ by Lemma \ref{m/q}) then $|H \cap \Omega_r| \leq |H \cap \soc(G)|$. Therefore if $r$ is a prime divisor of $m$ and $H_r \not \in \mathcal{M}$ then in order to cover $\Omega_r$ we need at least $$\frac{\min_r |\Omega_r|}{|H \cap \soc(G)|} \geq \frac{40 \cdot 360^{m-1}}{360^{m/2}} = 40 \cdot 360^{m/2-1}$$ subgroups. Since $m \geq 3$, this contradicts $\sigma(G) \leq \omega(2m) + 2 \cdot 6^m$. Therefore $\mathcal{K}_1 \subseteq \mathcal{M}$. If $m$ is even then $\mathcal{K}_1$ covers $G$, thus $\mathcal{K}_1 = \mathcal{M}$ and we are done. Suppose $m$ is odd. Since the subgroups of $S_6$ isomorphic to $S_5$ do not intersect $C$, the family $\mathcal{K}_1$ does not cover $\Omega_2$. Since $\Omega_2 \subset H_2$ and $\mathcal{K}_1 \cup \{H_2\}$ covers $G$, we obtain $\sigma(G) = |\mathcal{M}| = \omega(2m) + 2 \cdot 6^m$.

Let $n \geq 8$ be even. Let $\Pi$ be the set of the $n$-cycles in $S_n$, and let $\mathcal{A}$ be the family of the maximal imprimitive subgroups of $A_n$ corresponding to the partitions given by two subsets of $\{1,\ldots,n\}$ of size $n/2$. In \cite{maroti2} (claims 3.3 and 3.4) it is proved that if $n \geq 8$ then $\mathcal{A}$ is definitely unbeatable on $\Pi$. Condition (1) of Observation \ref{stim} follows from Lemma \ref{ab} and the fact that the order of a primitive maximal subgroup of $A_n$ is at most $2.6^n$ (see \cite{maroti}). In fact $(n/2)!^2 \geq 2.6^n$ if $n \geq 10$, and all the maximal subgroups of $A_8$ whose normalizers in $S_8$ contain $8$-cycles belong to $\mathcal{A}$. We now prove condition (2). We may assume that $m$ is not a power of $2$ by Lemma \ref{m/q}. If $n \in \{8,10\}$ then $|K|^{m-1} \cdot |N_{S_n}(K) \cap \Pi| \geq 2m |A_n|^{m/2}$ whenever $K \in \mathcal{A}$. Suppose $n \geq 12$. Using Lemma \ref{tremezz} we see that $|K|^{3/2} \geq |A_n|$ for every $K \in \mathcal{A}$. Therefore since $m \geq 2$ is not a power of $2$, if $p$ is the smallest prime divisor of $m$ then $|K|^{m-1} \geq |A_n|^{(2/3)(m-1)} \geq 2m|A_n|^{m/p}$ for $K \in \mathcal{A}$ (if $m \geq 5$ this follows from $p \geq 2$). Applying Observation \ref{stim} we obtain that

\begin{equation}
\left( \frac{1}{2} \binom{n}{n/2} \right)^m \leq \sigma(G) \leq \omega(2m) + \left( \frac{1}{2} \binom{n}{n/2} \right)^m + \sum_{i=1}^{[n/3]} \binom{n}{i}^m. \nonumber
\end{equation}

The upper bound is obtained by observing that the non-$n$-cycles of $S_n$ are covered by the maximal intransitive subgroups of $S_n$ of type $(i,n-i)$ for $1 \leq i \leq [n/3]$.

\section{Proof of Theorem \ref{a5wrc2}}

In this whole section we will call $G:=A_5 \wr C_2$, the semidirect product $(A_5 \times A_5) \rtimes \langle \varepsilon \rangle$ where $\varepsilon$, of order $2$, acts on $A_5 \times A_5$ exchanging the two variables. Recall that the maximal subgroups of $G$ are of the following five types:

\begin{itemize}
\item The socle $N=A_5 \times A_5$.
\item Type 'r': $N_G(M \times M^l)$ where $l \in A_5$ and $M$ is a point stabilizer.
\item Type 's': $N_G(M \times M^l)$ where $l \in A_5$ and $M$ is the normalizer of a Sylow $5$-subgroup.
\item Type 't': $N_G(M \times M^l)$ where $l \in A_5$ and $M$ is an intransitive subgroup of type $(3,2)$.
\item Type 'd': $N_G(\Delta_{\alpha})$ where $\alpha \in S_5$ and $\Delta_{\alpha}:=\{(x,x^{\alpha})\ |\ x \in A_5\}$.
\end{itemize}

Recall that:

\begin{itemize}
\item $N \cap N_G(H) = H$ for every $H$ of the type $M \times M^l$ or $\Delta_{\alpha}$ with $M$ a maximal subgroup of $A_5$.
\item The element $(x,y) \varepsilon$ belongs to $N_G(M \times M^l)$ if and only if $xl^{-1},ly \in M$. In particular $xy \in M$.
\item The element $(x,y) \varepsilon$ belongs to $N_G(\Delta_{\alpha})$ if and only if $(\alpha y)^2 = xy$.
\end{itemize}

Let $\mathcal{M}$ be a family of proper subgroups of $G$ which cover $G$.

\begin{obs} \label{insocle}
$N \in \mathcal{M}$
\end{obs}
\begin{proof}
Let $x \in A_5$ be a $5$-cycle, and let $y \in A_5$ be a $3$-cycle. Then the element $(x,y)$ does not belong to any $N_G(M \times M^l)$ or $N_G(\Delta_{\alpha})$ by the remarks above (no maximal subgroup of $A_5$ has order divisible by $3$ and $5$).
\end{proof}

Call $i$ the number of subgroups of type $i$ in $\mathcal{M}$ for $i=r,s,t,d$.

The 'type' of an element $(x,y)\varepsilon \in G-N$ is the cyclic structure of the element $xy \in A_5$. The four possible cyclic structures will be denoted by $1$, $(3)$, $(5)$, $(2,2)$.

The only maximal subgroups of $G$ containing elements of type $(3)$ are the ones of type $r$ or $t$ or $d$. A subgroup of type $r$ contains $96$ elements of type $(3)$. A subgroup of type $t$ contains $12$ elements of type $(3)$. A subgroup of type $d$ contains $20$ elements of type $(3)$. $G$ contains $1200$ elements of type $(3)$. In particular $96r+12t+20d \geq 1200$, in other words
\begin{equation} \label{est1}
24r+3t+5d \geq 300.
\end{equation}

The only maximal subgroups of $G$ containing elements of type $(5)$ are the ones of type $s$ or $d$. A subgroup of type $s$ contains $40$ elements of type $(5)$. A subgroup of type $d$ contains $24$ elements of type $(5)$ if $\alpha$ is even, $0$ if $\alpha$ is odd. $G$ contains $1440$ elements of type $(5)$. In particular $40s+24d \geq 1440$, in other words
\begin{equation} \label{est2}
5s+3d \geq 180.
\end{equation}

We know that $G$ admits a cover which consists of $57$ proper subgroups, with $s=36$, $r=20$, $t=d=0$ (the $20$ subgroups of type $r$ are $N_G(M \times M^l)$ where $l \in A_5$ and $M \in \{ \Stab(1), \Stab(2), \Stab(3), \Stab(4)\}$).

Suppose by contradiction that $\sigma(G)<57$, and let $\mathcal{M}$ be a cover with $56$ proper subgroups. In particular $r+s+t+d+1=56$, i.e. $r+s+t+d=55$.

\begin{obs}
$d \leq 33$, $s \geq 17$ and $r \geq 6$.
\end{obs}

\begin{proof}
Inequality \ref{est2} re-writes as $s \geq 36-\frac{3}{5}d$. Since $r+s+t+d=55$, $r+t = 55-s-d \leq 55-36+\frac{3}{5}d-d = 19-\frac{2}{5}d$. Combining this with inequality \ref{est1} we obtain $24(19-\frac{2}{5}d)+5d \geq 300$, i.e. $d \leq 156 \cdot 5/23$, i.e. $d \leq 33$. Therefore $s \geq 36-\frac{3}{5} d = 36-\frac{99}{5} > 16$.

Inequality \ref{est1} re-writes as $21r+2d-3s+3(r+t+d+s) \geq 300$, i.e. $21r+2d-3s \geq 135$. Since $d \leq 33$ and $s \geq 17$, $21r \geq 135+3 \cdot 17-2 \cdot 33 = 120$, i.e. $r \geq 6$.
\end{proof}

\begin{obs}
$r+t+d \geq 20$ and $s<36$.
\end{obs}

\begin{proof}
Consider the following elements of $A_5$: $a_1:=(243) \in \Stab(1)$, $a_2:=(143) \in \Stab(2)$, $a_3:=(142) \in \Stab(3)$, $a_4:=(132) \in \Stab(4)$. Let $\mathcal{X}$ be the set of elements of $G$ of the form $(x,y)\varepsilon$ with $xy=a_i$ for an $i \in \{1,2,3,4\}$ and $x \in J_i$, where $J_i$ is a fixed set of representatives of the right cosets of $\Stab(i)$, which will be specified later. Let $\mathcal{H}$ be the set of the $20$ subgroups $N_G(M \times M^l)$ of $G$ of type $r$ with $M$ the stabilizer of $i$ for $i \in \{1,2,3,4\}$. Notice that every element of $\mathcal{X}$ lies in exactly one element of $\mathcal{H}$. Now observe that if a subgroup $N_G(K \times K^l)$ of type $t$ contains an element $(x,y)\varepsilon \in \mathcal{X}$ then $K$ is determined by $a_i=xy$ - use this to label the $K$'s as $K_i$ for $i \in \{1,2,3,4\}$ -, so that the only freedom is in the choice of the coset $K_i l$. We will choose the sets $J_i$ in such a way that any two elements of $J_i$ lie in different right cosets of $K_i$. This implies that for every subgroup $N_G(K \times K^l)$ of $G$ of type $t$ we have $|\mathcal{X} \cap N_G(K \times K^l)| \leq 1$. Let us choose the $J_i$'s in such a way that for every subgroup $N_G(\Delta_{\alpha})$ of type $d$ we have $|\mathcal{X} \cap N_G(\Delta_{\alpha})| \leq 1$. Choose: $$J_1=\{(452),(12534),(13425),(14)(35),(23)(15)\},$$ $$J_2=\{(134),(245),(123),(152),(125)\},$$ $$J_3=\{(142),(132),(134),(153),(135)\},$$ $$J_4=\{(132),(142),(243),(154),(145)\}.$$We have that for any $i=1,2,3,4$ any two elements of $J_i$ lie in different right cosets of $K_i$. We have to check that every subgroup of the form $N_G(\Delta_{\alpha})$ contains at most one element of $\mathcal{X}$. In other words we have to check that if $(x,y)\varepsilon \in \mathcal{X} \cap N_G(\Delta_{\alpha})$ then $(x,y)\varepsilon$ is determined. We have $(\alpha y)^2 = xy$, so that if $\alpha$ is even then $\alpha = xyx$, if $\alpha$ is odd then $\alpha = \tau_{xy} xyx$, where $\tau_{xy}$ is the transposition whose support is pointwise fixed by $xy$. Let $$P_i:=\{xyx\ |\ xy=a_i,\ (x,y)\varepsilon \in \mathcal{X}\} \cup \{\tau_{xy} xyx\ |\ xy=a_i,\ (x,y)\varepsilon \in \mathcal{X}\} \subset S_5$$for $i=1,2,3,4$. Clearly $|P_i|=10$ for $i=1,2,3,4$. All we have to show is that the $P_i$'s are pairwise disjoint. This follows from the computation: $$P_1=\{(25)(34),(12)(35),(135),(14532),(15)(24),$$ $$(125)(34),(1352),(35),(132)(45),(24)\},$$ $$P_2=\{1,(15243),(14)(23),(14352),(14325),$$ $$(25),(1543),(14)(253),(1435),(1432)\},$$ $$P_3=\{(124),(14)(23),(234),(14253),(14235),$$ $$(124)(35),(14)(235),(2354),(1425),(1423)\},$$ $$P_4=\{(123),(13)(24),(12)(34),(13254),(13245),$$ $$(123)(45),(13)(245),(12)(345),(1325),(1324)\}.$$Clearly, the subgroups of $G$ of type $s$ do not contain any element of $\mathcal{X}$.

All this implies that $\mathcal{H}$ is definitely unbeatable on $\mathcal{X}$, hence $r+t+d \geq |\mathcal{H}| = 20$. It follows that $56 = |\mathcal{M}| = 1+r+s+t+d > s+20$, i.e. $s < 36$.
\end{proof}

\begin{obs} \label{n5}
Let $M$ be the normalizer of a Sylow $5$-subgroup of $A_5$, let $l \in A_5$ and suppose that $N_G(M \times M^l) \not \in \mathcal{M}$. Then $N_G(\Delta_{\alpha}) \in \mathcal{M}$ for every $\alpha \in Ml$. In particular if $\mathcal{L}$ is the family of the cosets $Ml$ where $M < A_5$ is the normalizer of a Sylow $5$-subgroup and $N_G(M \times M^l) \not \in \mathcal{M}$ then the number of subgroups of type $d$ in $\mathcal{M}$ is at least the size of the union of $\mathcal{L}$.
\end{obs}

\begin{proof}
The number of elements of type $(5)$ in $N_G(M \times M^l)$ is $40$. Moreover the only maximal subgroup of $G$ of type $r,s,t$ which contains one of these $40$ elements is the one we are considering: $xy \in M$ determines $M$ and $x \in Ml$ determines $Ml$. Let $c \in M$ be a $5$-cycle. The element $(x,x^{-1}c)\varepsilon$ belongs to $N_G(\Delta_{\alpha})$ if and only if $(\alpha x^{-1} c)^2=c$, i.e. $\alpha x^{-1} c = c^3$, i.e. $\alpha = c^2 x$. The result follows.
\end{proof}

\begin{lemma} \label{cosets}
We have the following facts:
\begin{enumerate}
\item Let $k$ be a positive integer, and let $\mathcal{L}$ be the family of the cosets of the normalizers of the Sylow $5$-subgroups of $A_5$. Then any subfamily of $\mathcal{L}$ consisting of exactly $k$ cosets covers at least $10k-2 \binom{k}{2}$ elements of $A_5$.

\item Let $H \neq K$ be two normalizers of Sylow $5$-subgroups of $A_5$. Then for any $a_1,a_2,a_3,b_1,b_2,b_3 \in A_5$ such that $Ha_1,Ha_2,Ha_3,Kb_1,Kb_2,Kb_3$ are pairwise distinct, the union $$Ha_1 \cup Ha_2 \cup Ha_3 \cup Kb_1 \cup Kb_2 \cup Kb_3$$has size at least $42$.
\end{enumerate}
\end{lemma}

\begin{proof}
Let $Ha,Kb \in \mathcal{L}$. If the intersection $Ha \cap Kb$ is non-empty then it contains an element $x$, so that $Ha=Hx$, $Kb=Kx$, and $Ha \cap Kb = Hx \cap Kx = (H \cap K)x$. It follows that the maximum size of the intersection of two elements of $\mathcal{L}$ equals the maximum size of the intersection of two normalizers of Sylow $5$-subgroups, i.e. $2$. Maximizing the sizes of the intersections we find that $k$ cosets cover at least $10k-2 \binom{k}{2}$ elements.

We now prove the second statement. Clearly $|Ha_1 \cup Ha_2 \cup Ha_3|=30$. Since $|Ha_i \cap Kb_j| \leq 2$ for every $i,j=1,2,3$, $$|Ha_1 \cup Ha_2 \cup Ha_3 \cup Kb_1 \cup Kb_2 \cup Kb_3| \geq 30+3 \cdot (10-3 \cdot 2) = 42,$$as we wanted.
\end{proof}

\begin{cor} \label{solved}
$s \leq 31$ and $d \geq 30$.
\end{cor}

\begin{proof}
Recall that the subgroups of $G$ of type $s$ are $36$. In the following we use Lemma \ref{cosets} and Observation \ref{n5}. If $s=32$ then $d \geq 28$, impossible; if $s=33$ then $d \geq 24$, impossible; if $s=34$ then $d \geq 18$, impossible since $r \geq 6$. Assume now $s=35$, so that $d \geq 10$. Since $r \geq 6$, $6+35+t+d \leq r+s+t+d = 55$, i.e. $t+d \leq 14$. Thus inequality \ref{est1} implies that $5 \cdot 14 \geq 300-24r$, i.e. $r \geq 10$. Hence $d=r=10$ and $t=0$. This contradicts inequality \ref{est1}. Since $s<36$, we deduce that $s \leq 31$ and consequently $d \geq 30$.
\end{proof}

Since $d \geq 30$, $r+s+t+30 \leq r+s+t+d = 55$, i.e. $r+s+t \leq 25$. Since $s \geq 17$ we obtain that $r+t \leq 8$. In particular $r \in \{6,7,8\}$.

\begin{itemize}
\item $r=6$. Then by inequality \ref{est1} we have $144+5(t+d) \geq 24r+3t+5d = 24r+3t+5d \geq 300$, and we deduce that $t+d \geq 32$. Therefore $55=r+s+t+d \geq 6+s+32$, i.e. $s \leq 17$. Since $s \geq 17$ we obtain that $s=17$. Inequality \ref{est2} says that $5 \cdot 17 + 3d \geq 180$, i.e. $d \geq 32$, so that $d=32$ and $t=0$.
\item $r=7$. Since $d \geq 30$, $7+s+30 \leq r+s+t+d=55$, i.e. $s \leq 18$.
\begin{itemize}
\item $s=18$. Then $7+18+t+d=r+s+t+d=55$, i.e. $t+d=30$. Since $d \geq 30$ we obtain $d=30$ and $t=0$.
\item $s=17$. Inequality \ref{est2} says that $5 \cdot 17 + 3d \geq 180$, i.e. $d \geq 32$, so that $55=r+s+t+d \geq 7+17+32 = 56$, contradiction.
\end{itemize}
\item $r=8$. Then since $r+s+t \leq 25$ we obtain $s+t \leq 17$, and since $s \geq 17$ we have $s=17$, $t=0$ and $d=30$. This contradicts inequality \ref{est2}.
\end{itemize}

We deduce that either $(r,s,t,d) = (7,18,0,30)$ or $(r,s,t,d) = (6,17,0,32)$.

In both these cases there are at least $18$ subgroups of type $s$ outside $\mathcal{M}$. Therefore Observation \ref{n5} and Lemma \ref{cosets}(2) imply that $d \geq 42$, a contradiction.

\end{document}